\documentclass[10pt]{amsart}
\usepackage{amsmath,amssymb,amscd,mathrsfs,enumerate}

\numberwithin{equation}{section}

\newtheorem{theorem}[equation]{Theorem}

\newtheorem{lemma}[equation]{Lemma}

\theoremstyle{definition}
\newtheorem{definition}[equation]{Definition}

\DeclareMathOperator{\Char}{Char}
\DeclareMathOperator{\Diff}{Diff}
\DeclareMathOperator{\Levi}{Levi}
\DeclareMathOperator{\sym}{ \sigma\!\!\!\sigma}
\DeclareMathOperator{\Res}{Res}
\DeclareMathOperator{\Span}{span}
\DeclareMathOperator{\spec}{spec}
\DeclareMathOperator{\Tr}{Tr}
\DeclareMathOperator{\WF}{WF}

\def\A{\mathcal A}
\def\B{\mathcal B}
\def\D{\mathcal D}
\def\Dom{\mathscr D}
\def\Dbar{\overline\D}
\def\Hor{\mathcal H}
\def\K{\mathcal K}
\def\Kbar{\overline \K}
\def\Lie{\mathcal {L}}
\def\M{\mathcal M}
\def\N{\mathcal N}
\def\Orb{\mathcal O}
\def\T{\mathcal T}
\def\Vee{\mathcal V}
\def\Veebar{\overline\Vee}
\def\X{\mathcal X}
\def\Y{\mathcal Y}

\def\Ha{\mathscr H}
\def\Sch{\mathscr S}

\def\C{\mathbb C}
\def\Dee{\mathbb D}
\def\Deebar{\overline\Dee}
\def\R{\mathbb R}
\def\Z{\mathbb Z}

\def\a{\mathfrak a}
\def\m{\mathfrak m}
\def\p{\mathfrak p}

\def\ss{\mathfrak s}

\def\im{i}
\def\Wedge{\raise2ex\hbox{$\mathchar"0356$}}

\def\inner{\mathbf i}
\def\Id{I}

\def\minus{\backslash}

\def\ie{i.e.}
\def\dee{\partial}
\def\deebar{\overline \partial}
\def\deebarb{\overline \partial_b}
\def\Laplacian{\square}
\def\embed{\hookrightarrow}

\def\display#1#2{\mbox{\parbox{#1} {#2}}}
\def\set#1{\{#1\}}

\begin{document}

\title{Hypoellipticity and vanishing theorems}
\author{Gerardo A. Mendoza}
\email{gmendoza@math.temple.edu}
\address{Department of Mathematics\\
Temple University\\
Philadelphia, PA 19122}
\begin{abstract}
Let $-\im\Lie_\T$ (essentially Lie derivative with respect to $\T$, a smooth nowhere zero real vector field) and $P$ be commuting differential operators, respectively of orders $1$ and $m\geq 1$, the latter formally normal, both acting on sections of a vector bundle over a closed manifold. It is shown that if $P+(-i\Lie_\T)^m$ is elliptic then the restriction of $-\im\Lie_\T$ to $\Dom\subset \ker P\subset L^2$ yields a selfadjoint  operator $-\im\Lie_\T|_\Dom:\Dom\subset\ker P\to \ker P$ with compact resolvent ($\Dom$ is specified carefully). It is also shown that, in the presence of an additional hypothesis on microlocal hypoellipticity of $P$, $-\im\Lie_\T|_\Dom$ is semi-bounded. These results are applied to CR manifolds on which $\T$ acts as an infinitesimal CR transformation which are then shown to yield versions of Kodaira's vanishing theorem.
\end{abstract}


\dedicatory{To Mar\'ia Silvia}
\maketitle

\section{Introduction}

The main results in this paper were motivated by an investigation into properties of complex $b$-structures. The latter, introduced in \cite{Me3}, are complex structures in the $b$-category (see Melrose \cite{RBM2}) on manifolds $\M$ with boundary. Complex $b$-structures happen to determine a very rich structure on the boundary of $\M$ bearing much similarity with the structure of a circle bundle of a holomorphic line bundle over a complex manifold (the last section here goes into this in much detail). While they are not CR structures, they do contain families of these in the same way that a circle bundle of a holomorphic line bundle admits a family of CR structures parametrized by Hermitian structures through the Hermitian holomorphic connection. Various aspects of these structures on the boundary were investigated in depth in a series of papers \cite{Me4,Me6,Me7,Me8} going further into properties motivated by those of circle bundle. This paper represents another investigation along those lines, this time in the form of theorems about vanishing of cohomology. We will not discuss here complex $b$-structures, but refer the interested reader to any of the papers already cited. 

\medskip
Throughout this paper, $\N$ will denote a $C^\infty$ compact manifold without boundary, $\T$ a smooth nowhere vanishing real vector field, and $E\to\N$ a complex Hermitian vector bundle. Let $\Lie_\T$ be a first order differential operator acting on sections of $E$ related to $\T$ by the property
\begin{equation}\label{LieAsCovariantDerivative}
\Lie_\T(f\phi)=f\Lie_\T\phi+\T\! f\,\phi,\quad f\in C^\infty(\N),\ \phi\in C^\infty(\N;E),
\end{equation}
such that $-\im \Lie_\T$ is symmetric (the $L^2$ inner product is defined with the aid of the Hermitian form of $E$ and a $\T$-invariant smooth positive density). Suppose $P$ is a differential operator that commutes with its formal adjoint and with $\Lie_\T$. We show in Section~\ref{sInvariantOperators}, see Theorem~\ref{TIsFredholm}, that if $P+(-\im\Lie_\T)^m$ is elliptic, then $-\im\Lie_\T$, acting on a subspace $\Dom$ of the kernel of $P$ in $L^2$, is selfadjoint with compact resolvent. The domain for $-\im\Lie_\T$ that makes the statement precise will be specified in \eqref{TheDomain}.

While the operator $-\im\Lie_\T\big|_\Dom$ acting on $\Dom$ is not, strictly speaking, a differential (or pseudodifferential) operator any longer, it inherits many properties from these, enough that one can prove a rough estimate on the counting function of its eigenvalues, see \eqref{RoughWeyl}; this can be done without hypothesis beyond those already stated in Theorem~\ref{TIsFredholm}. However, assuming in addition positivity of $A=P+(-\im\Lie_\T)^m$ one can give a specific upper bound for the counting function using Weyl's estimate for $A$; this is the content of Theorem~\ref{WeylBound}. The relevancy of this lies in its implication on the growth of the dimension of the spaces of holomorphic sections of an ample line bundle which we do not discuss here. These results, together with a rough outline of the proof of Weyl's estimate is presented at length in Section~\ref{sWeyl}. 

Semi-boundedness of $-\Lie_\T\big|_\Dom$ in the presence of hypoellipticity conditions on $P$ is discussed in Section~\ref{sSpectrumHypo}. The ellipticity of $P+(-\im\Lie_\T)^m$ implies that the characteristic set $\Char(P)$ of $P$ lies in the complement of the set where the principal symbol, $\pmb \tau$ (a scalar function), of $-\im \T$  vanishes. Consequently $\Char P$ is separated into two subsets $\Char^{\pm}(P)$, according to the sign of $\pmb\tau$. Theorem~\ref{HalfSpectrum} states that if, for instance, $P$ is hypoelliptic on $\Char^+(P)$ then $-\im\Lie_\T\big|_\Dom$ has only finitely many positive eigenvalues. This is the central result concerning vanishing theorems.

The previous theorems are applied in Sections~\ref{sCRmanifolds} and \ref{sVanishing} to structures of the kind arising on the boundary of a complex $b$-manifold. In this paper we take the point of view that there is an initially given CR structure on $\N$ which is invariant under the action of the one-parameter group of diffeomorphisms generated by $\T$ and construct the aforementioned additional structure part of the way form this, enough that the analogy with line bundles (discussed in Section~\ref{sLineBundles}) becomes clear. All CR structures in the class are again $\T$-invariant. The class is analogous to the class of Hermitian holomorphic connections on a holomorphic line bundle parametrized by the Hermitian metric. Section~\ref{sCRmanifolds} ends with a restatement of Theorem~\ref{TIsFredholm}, which is now a decomposition theorem of the $L^2$-CR cohomology  according to the eigenspaces of $-\im\Lie_\T$ acting as Lie derivative on the spaces of harmonic CR forms. Throughout Sections ~\ref{sCRmanifolds} and \ref{sVanishing} we work under the assumption that there is a $\T$-invariant metric which then in particular gives $\T$-invariant Hermitian structures on all CR structures of interest.

The theorem relating hypoellipticity and semi-finiteness of the spectrum in Section~\ref{sSpectrumHypo} does not make any assumption about where the hypoellipticity comes from. In Section~\ref{sVanishing} we use known theorems (see for instance Boutet de Monvel \cite{BdM} or Sj\"ostrand \cite{Sj}) that establish microlocal hypoellipticity in the presence of the hypothesis of non-degeneracy of the CR structure to state a theorem concerning nature of the decomposition of the spaces of harmonic CR forms as eigenspaces of $-\im\Lie_\T$. Incidentally, a complete discussion of hypoellipticity of the Laplacian on CR forms can be found in the appendix of \cite{Me3} as part of a complete symbol calculus for a class of pseudodifferential operator that contains these Laplacians when the CR structure is non-degenerate.

We have included, as Section~\ref{sLineBundles}, a discussion of circle bundles of line bundles that provides a translation of known points of view to that of the present paper. This serves to give concrete examples to the theorems discussed here, in particular the relation between spectrum, eigenspaces, and cohomology, see \eqref{SpectrumAndKodairaVanishing}. 

\medskip
This paper contains the results presented by the author in the Workshop on Several Complex Variables and Complex Geometry held in the Academia Sinica, Taipei from July 9 to July 13, 2012. The author thanks the organizers for the opportunity to participate in the event.

\section{Invariant operators}\label{sInvariantOperators}

We let $\a_t$ denote the one-parameter family of diffeomorphisms determined by $\T$. Let $E\to\N$ be a vector bundle with Hermitian metric $h$ and suppose that
\begin{equation}\label{LieT}
\Lie_\T:C^\infty(\N;E)\to C^\infty(\N;E)
\end{equation}
is a differential operator such that
\begin{equation}\label{InvariantHMetric}
\T h(\phi,\psi)=h(\Lie_\T\phi,\psi)+h(\phi,\Lie_\T\psi)
\end{equation}
holds if $\phi$, $\psi\in C^\infty(\N;E)$. Such an operator must satisfy \eqref{LieAsCovariantDerivative}
for every smooth $f$ and section $\phi$ so it must be a first order differential operator. It can be viewed as the operator $\nabla_\T$ for some Hermitian connection $\nabla$. Indeed, if $\nabla'$ is an arbitrary Hermitian connection on $E$ and $\theta$ is a smooth real $1$-form such that $\langle \theta,\T\rangle = 1$, then
\begin{equation*}
\phi\mapsto \nabla\phi = \nabla'\phi + \theta\otimes \Lie_\T \phi - \theta\otimes \nabla'_\T\phi
\end{equation*}
is a Hermitian connection with the required property. So $\Lie_\T$ gives rise to a one-parameter group of isometries $\a_t^*:E\to E$ covering $\a_{-t}:\N\to\N$ by way of parallel transport along the integral curves of $\T$. Conversely, the infinitesimal generator of such a group of isometries is an operator \eqref{LieT} for which \eqref{InvariantHMetric} holds. It follows immediately from \eqref{LieAsCovariantDerivative} that $\sym(-\im \Lie_\T)=\sym(-\im\T)\Id$. 

Let $\m$ be a smooth positive density on $\N$ and define the space $L^2(\N;E)$ using the Hermitian metric of $E$ and the density $\m$; the inner product is thus
\begin{equation*}
(\phi,\psi)=\int h(\phi,\psi)\,d\m.
\end{equation*}

Let $P\in\Diff^m(\N;E)$, $m\geq 1$. By $\ker P$ we shall mean mean the kernel of $P$ in $L^2(\N;E)$; as a closed subspace of $L^2(\N;E)$, it is a Hilbert space on its own right. Define
\begin{equation}\label{TheDomain}
\Dom =\set{\phi\in \ker P:\Lie_\T\phi\in L^2(\N;E)}.
\end{equation}
If $P$ commutes with $\Lie_\T$, then
\begin{equation}\label{LieInL2}
-\im \Lie_\T\big|_\Dom:\Dom \subset \ker P\to\ker P
\end{equation}
is an unbounded closed operator.

\begin{theorem}\label{TIsFredholm}
Suppose that $P\in\Diff^m(\N;E)$ commutes with its formal adjoint and with $\Lie_\T$ and that there is a (real) line $\Lambda\subset \C$ through the origin such that
\begin{equation}\label{DoubleRayCondition}
\sym (P)+\sym(-\im\Lie_\T)^m-\lambda\Id \text{ is invertible if } \lambda\in \Lambda.
\end{equation}
Suppose further that the Hermitian metric of $E$ and the density $\m$ are $\T$-invariant. Then the operator \eqref{LieInL2} is selfadjoint with compact resolvent, in particular, Fredholm.
\end{theorem}

If $P$ is symmetric then the principal symbol of $P+(-\im\Lie_\T)^m$ is selfadjoint. So  \eqref{DoubleRayCondition} holds if this opera elliptic for any line $\Lambda$ different from the real axis. since in this case the principal symbol of $P+(-\im\Lie_\T)^m$ is selfadjoint.

\medskip
Theorem~\ref{TIsFredholm} is a general version of Theorem 7.5 in \cite{Me7}. The following proof is adapted from that paper.

\begin{proof}
First we note that $-\im \Lie_\T$ is symmetric on $H^1(\N;E)$, the $L^2$-based Sobolev space of order $1$. Indeed, if $\phi$, $\psi\in C^\infty(\N;E)$, then \eqref{InvariantHMetric}
gives
\begin{equation*}
\int \T h(\phi,\psi)\,\m = (\Lie_\T\phi,\psi)+(\phi,\Lie_\T\psi).
\end{equation*}
On the other hand, it follows from the $\T$-invariance of $\m$ that if $u$ is a smooth function, then $(\T u)\, \m=d(u\,\inner_\T\m)$ where $\inner_\T$ is interior multiplication. So the integral on the left vanishes and we get that the formal adjoint of $\Lie_\T$ is $-\Lie_\T$. Using that $C^\infty(\N;E)$ is dense in $H^1(\N;E)$ we get
\begin{equation*}
(-\im \Lie_\T\phi,\psi)=(\phi,-\im \Lie_\T\psi),\quad \phi,\ \psi\in H^1(\N;E).
\end{equation*}

We now show that $\Dom\subset H^1(\N;E)$. Since $0\in \Lambda$, \eqref{DoubleRayCondition} implies that
\begin{equation*}
A=P+(-\im\Lie_\T)^m
\end{equation*}
is elliptic. Let $Q$ be a parametrix for $A$, so that
\begin{equation*}
QA=\Id-R
\end{equation*}
where $R$ is a smoothing operator. If $\phi\in \ker P$, then
\begin{equation*}
\phi=Q(-\im\Lie_\T)^m \phi + R\phi.
\end{equation*}
Suppose $\phi\in\Dom$. Since $\Lie_\T\phi\in L^2(\N;E)$ and $S=Q(-\im\Lie_\T)^{m-1}$ is a classical pseudodifferential operator of order $-1$, $S\Lie_\T \phi\in H^1(\N;E)$, and since $R\phi\in C^\infty(\N;E)$, $\phi\in H^1(\N;E)$. Consequently \eqref{LieInL2} is a symmetric operator.

Let $\pmb \tau=\sym(-\im \T)$, so that $\sym(-\im \Lie_\T)=\pmb \tau\Id$ as we already noted. If $\Lambda$ is the real axis, then setting $\lambda=\pmb \tau(\nu)^m$ in \eqref{DoubleRayCondition} gives that $\sym(P)(\nu)$ itself is invertible at any $\nu\in T^*\N\minus 0$, \ie, $P$ is elliptic. Then $\ker P$ is finite-dimensional and consists of smooth sections, so $\Dom=\ker P$, and $-\im\Lie_\T$ is selfadjoint. So assume that $\Lambda$ is not the real axis.

As is well known (Seeley \cite{Seeley}), \eqref{DoubleRayCondition} implies that $(A-\lambda):H^m(\N;E)\to L^2(\N;E)$ is invertible for each $\lambda\in \Lambda$ with sufficiently large $|\lambda|$. For such $\lambda$, the inverse, $Q_\lambda$, is a pseudodifferential operator of order $-m$. It commutes with $\Lie_\T$ and $P$ since $\Lie_\T$ commutes with $A$. The formula
\begin{equation*}
\big((-\im\Lie_\T)^m-\lambda\big)Q_\lambda=I- P Q_\lambda,
\end{equation*}
valid on $L^2(\N;E)$, gives
\begin{equation}\label{RInverse}
\big((-\im\Lie_\T)^m-\lambda\big)Q_\lambda\phi =\phi \quad \text{if }\phi \in \ker P,
\end{equation}
whereas the formula
\begin{equation*}
Q_\lambda\big(P+(-\im\Lie_\T)^m-\lambda\big)=I,
\end{equation*}
valid on $H^1(\N;E)$, gives
\begin{equation}\label{LInverse}
Q_\lambda((-\im\Lie_\T)^m-\lambda)\phi =\phi \quad \text{if }\phi \in \Dom.
\end{equation}
Let
\begin{equation*}
S_\lambda=Q_\lambda\sum_{j=0}^{m-1}\lambda^{m-j-1}(-\im\Lie_\T)^j,
\end{equation*}
a pseudodifferential operator of order $-1$, hence compact. Its restriction to $\ker P$ has range in $\ker P$ (because $P$ commutes with $\Lie_\T$ and $Q_\lambda$), hence in $\Dom$.
Let
\begin{equation*}
\Pi:L^2(\N;E)\to L^2(\N;E),\quad \iota:\ker P\to L^2(\N;E)
\end{equation*}
be respectively, the orthogonal projection on $\ker P$ and the inclusion map. Then $\hat S_\lambda=\Pi S_\lambda\iota:\ker P\to\ker P$ is compact. The formulas \eqref{RInverse}, \eqref{LInverse} give that $\hat S_\lambda$ is the inverse of
\begin{equation*}
(-\im \Lie_\T-\lambda)\big|_\Dom :\Dom\subset \ker P \to \ker P
\end{equation*}
for each $\lambda\in \Lambda$ with sufficiently large modulus.

We now show that $\Dom$ is dense in $\ker P$. Let $\psi\in \ker P$ be orthogonal to $\ker P$. If $\phi\in \ker P$ then $\hat S_\lambda\phi\in \Dom$, so $0=(\hat S_\lambda \phi,\psi) = (\phi,\hat S_\lambda^*\psi)$, and therefore $\psi\in \ker \hat S_\lambda^*$. Since $S_\lambda$ is continuous, $\hat S_\lambda^*=\pi S_\lambda \iota$. Since $P$ commutes with its formal adjoint and with $\Lie_\T$, so does $A$. This implies that $Q_\lambda^*$ commutes with $A$ and $\Lie_\T$, hence with $P$. Thus $Q_\lambda^*$ maps $\ker P$ to itself, and so does $S_\lambda^*$. Therefore $\hat S_\lambda \psi=0$ is equivalent to $S_\lambda^*\psi=0$. Since $(-\im\Lie_\T -\overline \lambda)S_\lambda\psi=Q_\lambda^*\psi$ and since $Q_\lambda^*$ is injective, $\psi=0$.

It follows that the operator \eqref{LieInL2} is densely defined, and since it is symmetric with resolvent set containing points in both components of $\C\minus \R$ (that is, its deficiency indices vanish), it is selfadjoint. Finally, since $\hat S_\lambda$ is compact, \eqref{LieInL2} is Fredholm.
\end{proof}

\section{Weyl estimates}\label{sWeyl}

Suppose that the conditions of Theorem~\ref{TIsFredholm} are satisfied and let $\spec_0(-\im\Lie_\T)$ denote the spectrum of the selfadjoint operator \eqref{LieInL2}. This is a discrete subset of $\R$ without finite points of accumulation. The eigenspaces, 
\begin{equation*}
\mathcal E_\tau=\set{\phi\in C^\infty(\N;E):P\phi=0,\ \Lie_\T\phi=\im\tau\phi},
\end{equation*}
are finite-dimensional and consist of smooth sections of $E$ because $P+(-\im\Lie_\T)^m$ is elliptic. We discuss here estimates for
\begin{equation*}
N(\tau)=\sum_{\substack{\tau_0\in\spec_0{(-\im\Lie_\T)}\\|\tau_0|<\tau}}\dim\mathcal E_{\tau'}
\end{equation*}
This is not quite the same as Weyl estimates for differential (or pseudodifferential) operators because \eqref{LieInL2} is not quite a differential operator. 

A rough estimate of the form 
\begin{equation}\label{RoughWeyl}
N(\tau) \leq C\tau^\mu
\end{equation}
for some positive numbers $C$ and $\mu$ can be obtained by the argument in the proof in Gilkey \cite[Lemma 1.6.3, part (c)]{Gilk84} (the proof of Lesch \cite[Proposition 1.4.7]{Lesch1997} is perhaps more explicit). The argument, which we shall omit while referring the reader to the just mentioned works, requires pointwise estimate of the elements of an orthonormal basis consisting of eigenvectors of \eqref{LieInL2} along the lines of the following result:
\begin{lemma}
Let $\{\phi_j\}_{j\in J}$ be an orthonormal basis of $\ker P$ consisting of eigenvectors of $-\im \Lie_\T$, $\phi_j\in \mathcal E_{\tau_j}$. Then there are positive constants $C$ and $\mu$ such that
\begin{equation}\label{PolynomialPointEigenBounds}
|\phi_j(p)|_{E_p}\leq C(1+|\tau_j|)^\mu\quad \text{ for all }p\in \N,\ j\in J.
\end{equation}
If $\psi \in C^\infty(\N;E^*\otimes |\Wedge|\N)$, then for each positive integer $N$ there is $C_N$ (depending on $\psi$) such that
\begin{equation}\label{BoundOnCoefficients}
|\langle \phi_j,\psi\rangle|\leq C_N(1+|\tau_j|)^{-N}\quad\text{ for all }j.
\end{equation}
\end{lemma}
The proof is a simple adaptation of that of \cite[Lemma 7.9]{Me7}.

The virtue of \eqref{RoughWeyl} lies in that it makes no assumptions on $P$ or $m$ other than the ones in Theorem~\ref{TIsFredholm}. Taking advantage of the standard Weyl estimate for positive elliptic operators the estimate can of course be made more precise:

\begin{theorem}\label{WeylBound}
In addition to the assumptions of Theorem~\ref{TIsFredholm}, suppose $A=P+(-\im\Lie_\T)^m$ is positive. Then
\begin{equation}\label{WeylFor-iLie_T}
N(\tau)\lesssim w_A|\tau|^{\dim\N}\ \text{as }\tau\to\infty,
\end{equation}
where $w_A$ is the coefficient in Weyl's eigenvalue estimate for $A$.
\end{theorem}

In formula \eqref{WeylFor-iLie_T}, $\lesssim$ means modulo an error of order $o(\tau^{\dim\N})$, $\tau\to\infty$. Weyl's estimate for the counting function of the eigenvalues of $A$ is
\begin{equation*}
\sum_{\lambda'<\lambda}\dim\ker(A-\lambda'\Id)\sim w_A\lambda^{\dim\N/m}\text{ as }\lambda\to\infty.
\end{equation*}
If $\tau$ is an eigenvalue of \eqref{LieInL2}, then $\tau^m$ is an eigenvalue of $A$ and $\mathcal E_\tau\subset \ker (A-\tau^m\Id)$. A simple argument now yields \eqref{WeylFor-iLie_T}. It should be noted that the hypothesis that $A$ is positive implies $m$ even and $P$ non-negative. 

Perhaps the simplest way to obtain $w_A$ is from the expansion at $t=0$ of the trace of the heat kernel of $A$ via the zeta function and the Wiener-Ikehara Tauberian Theorem \cite[Theorem XVII]{Wiener}. This is well known but we will briefly review here the less technical aspects for the sake of completeness. 

Let $e^{-tA}$ be the operator giving the solution of 
\begin{equation*}
\frac{\partial u}{\partial t}+Au=0,\ u\big|_{t=0}=u_0.
\end{equation*}
The operator $e^{-tA}$ is has smooth Schwartz kernel in $t>0$ so it is trace class for each positive $t$. If $\psi_k$ is an orthonormal basis of $L^2(\N;E)$ consisting of eigenvectors of $A$, $A\psi_k=\lambda_k$, then the solution operator is of course
\begin{equation*}
\sum_k e^{-t\lambda_k}(u_0,\psi_k)\psi_k
\end{equation*}
and so 
\begin{equation*}
\Tr e^{-tA} = \sum_{k=0}^\infty e^{-\lambda_kt}.
\end{equation*}
The zeta function of $A$ is 
\begin{equation*}
\zeta_A(s)=\sum_{k=0}^\infty \lambda_k^{-s},\quad \Re s\gg 0
\end{equation*}
which can also be written as
\begin{equation}\label{GenericZeta}
\zeta(s) = \frac{1}{\Gamma(s)}\int_0^\infty t^{s}\Tr e^{-tA} \,\frac{d t}{t}.
\end{equation}
On the other hand, by means of pseudodifferential techniques one obtains
\begin{equation*}
\Tr e^{-tA}\sim \sum_{k=0}^\infty a_k t^{(k-\dim\N)/m} \text{ as }t\to 0^+.
\end{equation*}
(see Grubb \cite[Corollary 4.2.7]{Grubb}) where the $a_k$ are numbers. Using  \eqref{GenericZeta} gives, with $\omega\in C_c^\infty(\R)$, $\omega=1$ near $0$, and the notation
\begin{equation*}
r_{K+1}(t)=\Tr e^{-tA} - \sum_{k=0}^K a_k t^{(k-\dim\N)/m}
\end{equation*}
gives
\begin{multline}\label{MellinTTrace}
\Gamma(s)\zeta_A(s)
= \int_0^\infty \omega(t)\sum_{k=0}^K  a_kt^{s+(k-\dim\N)/m} \,\frac{d t}{t}\\+\int_0^\infty t^{s}\omega(t)r_{K+1}(t) \,\frac{d t}{t}
+ \int_0^\infty (1-\omega(t))t^{s}\Tr e^{-tA}\,\frac{d t}{t}
\end{multline}
The first integral is equal to
\begin{equation*}
\sum_{k=0}^K  a_k\widehat\omega(s+(k-\dim\N)/m)
\end{equation*}
with
\begin{equation*}
\widehat\omega(s)=\int_0^\infty t^s\omega(t)\,\frac{dt}{t}.
\end{equation*}
This is a meromorphic function on $\C$ with a simple pole only at $0$ and residue $1$ there. Indeed, using integration by parts one gets
\begin{equation*}
\widehat\omega(s) = \frac{1}{s}\int_0^\infty t^s\omega'(t)\,dt;
\end{equation*}
the function defined by the integral is an entire function of $s$ with value $1$ at $s=0$. The second integral one the right hand side of \eqref{MellinTTrace} is holomorphic in $\Re s> (\dim\N-K-1)/m$, while the third is entire. So $\Gamma(s)\zeta_A(s)$ is meromorphic in $\C$ with simple poles at the points $(\dim\N-k)/m$, $k\in \mathbb N_0$ and residue $a_k$ there. One reads off from this the poles (all simple) and residues of $\zeta_A$ (in particular points of $-\mathbb N_0$ are not poles). For Weyl's asymptotic formula, the presence of the pole at $\dim\N/m$ with residue $a_0/\Gamma(\dim\N/m)$ is the only important information: By the Wiener-Ikehara Tauberian Theorem cited above,
\begin{equation*}
w_A=\frac{m}{\dim \N}\Res_{s=\dim\N/m}(\zeta_A(s)).
\end{equation*}

A by-product of the estimate \eqref{RoughWeyl} (whether the rough estimate or \eqref{WeylFor-iLie_T}) and the estimates \eqref{BoundOnCoefficients} give:
\begin{lemma}\label{FourierExpansion}
Let $\{\phi_j\}_{j\in J}$ be an orthonormal basis of $\ker P$ consisting of eigenvectors of $-\im \Lie_\T$. Then $\psi\in \ker P\cap C^\infty(\N;E)$ if and only if $(\phi,\psi_j)$ is rapidly decreasing in $j$:
\begin{equation*}
\text{for all }N>0 \text{ there is }C_N\text{ such that } |(\phi,\psi_j)|<C_N(1+j)^{-N}\text{ for all }j.
\end{equation*}
Further, if $\psi$ is smooth, then the Fourier series
\begin{equation*}
\psi=\sum_{j\in J}(\psi,\phi_j)\phi_j
\end{equation*}
converges in $C^\infty(\N;E)$.
\end{lemma}
This lemma holds because of the polynomial relation between the eigenvalues $\tau$ and the dimension of $\mathcal E_\tau$ as $\tau\to \infty$. 

\section{Hypoellipticity and spectrum}\label{sSpectrumHypo}

We continue to assume that the conditions of Theorem~\ref{TIsFredholm} are satisfied. 
The ellipticity of $P+(-\im\Lie_\T)^m$ implies that $\Char(P)\subset \set{\sym(-\im\T)\ne 0}$. Define
\begin{equation*}
\Char^{\pm}(P)=\set{\nu \in\Char(P):\sym(-\im\T)(\nu)\gtrless 0}.
\end{equation*}
Define also
\begin{equation*}
\spec_0^\pm(-\im \Lie_\T)=\set{\tau\in \spec_0(-\im \Lie_\T):\tau\gtrless 0}.
\end{equation*}

\begin{theorem}\label{HalfSpectrum}
Suppose that the hypotheses of Theorem~\ref{TIsFredholm} are satisfied and that $P$ is microlocally hypoelliptic at $\Char^+(P)$. Then $-\im\Lie_\T\big|_\Dom$ is semi-bounded from above, that is, $\spec_0^+(-\im\Lie_\T)$ is finite.
\end{theorem}

Of course the analogous statement for $\Char^-(P)$ and semi-boundedness from below of $-\im\Lie_\T\big|_\Dom$ also holds.

\medskip
The proof requires some preparation. Let $\p:\R\times\N\to\N$ and $\ss:\R\times \N\to\R$ be the canonical projections, and let $\Lie_{\partial_\ss}$ be the infinitesimal generator of the group of translations of $\p^*E$ in the direction of the fibers of $\p$.  Literally
\begin{equation*}
\p^*E=\set{(\ss,\p;\varphi):\varphi\in E_p}
\end{equation*}
so the meaning of translation in direction of the fibers of $\p$ is clear. Any differential operator $P$ on $C^\infty(\N;E)$ has a canonical lifting as a differential operator $\tilde P$ on $C^\infty(\R\times\N;\p^*E)$, characterized by the properties
\begin{equation*}
\p^*P=\tilde P\p^*,\quad \Lie_{\partial_\ss}\tilde P=\tilde P\Lie_{\partial_\ss},\quad \tilde P\ss =\ss \tilde P,
\end{equation*}
where $\ss$ in the last condition means the operator of multiplication by the real-valued function $\ss$. Define $\tilde \a_t:\R\times\N\to\R\times \N$ by
\begin{equation*}
\tilde \a_t(\ss,p)=(\ss,\a_t(p)).
\end{equation*}
This is the one-parameter group of diffeomorphisms whose infinitesimal generator is the canonical lifting, $\tilde \T$, of $\T$.

Define $\tilde \a_t^*:\p^*E\to\p^*E$ by
\begin{equation*}
\tilde \a_t^*(\ss, p;\varphi)=(\ss,\a_{-t}p;\a_t^*\varphi),
\end{equation*}
a one-parameter group of isomorphisms on $\p^*E$ covering $\tilde\a_{-t}$. Define
\begin{equation*}
\A:C^\infty(\N;E)\to C^\infty(\R\times\N;\p^*E)
\end{equation*}
by
\begin{equation*}
\A \phi(\ss,p)=(\ss,p;\a_\ss^*(\phi(\a_\ss(p)))).
\end{equation*}
This map has an extension to a continuous map $C^{-\infty}(\N;E)\to C^{-\infty}(\R\times\N;\p^*E)$. The formula
\begin{equation*}
\Lie_{\partial_\ss}\A\phi=\A\Lie_\T\phi, \quad\phi\in C^{-\infty}(\N;E)
\end{equation*}
holds since it holds for smooth $\phi$. A direct computation also gives that
\begin{equation*}
\A\a_t^*\phi=\tilde \a_t^*\A\phi,\quad\phi\in C^{-\infty}(\N;E),
\end{equation*}
so $A\Lie_\T\phi=\Lie_{\tilde\T}\A\phi$, which gives
\begin{equation}\label{FreeIdentity}
(\Lie_{\partial_\ss}-\Lie_{\tilde\T})\A\phi=0.
\end{equation}

Let $\Sch(\R\times\N;\p^*E)$ be the space of Schwartz sections of $\p^*E$, that is, the subspace of $C^\infty(\R\times \N;\p^*E)$ whose elements satisfy
\begin{equation*}
\forall k,\ell,m\in \mathbb N_0\, \forall P\, \in \Diff(\N;E)\,\exists C\text{ such that }\|\ss^k\Lie_{\partial_\ss}^\ell \tilde P\phi\|_{L^\infty}\leq C
\end{equation*}
where the norm is computed using the lifting of the Hermitian metric of $E$. Let $|\Wedge|(\R\times\N)$ be the density bundle of $\R\times \N$. Using $|d\ss|\otimes \p^*\m$ to trivialize the density bundle of $\R\times\N$, define
$\Sch'(\R\times \N;\p^*E)$ as the dual of $\Sch(\R\times\N;\p^*E^*)$.
As usual $\Sch(\R\times\N;\p^*E)\embed \Sch'(\R\times\N;\p^*E)$ is continuous with dense image.

If $\phi\in \Sch(\R\times\N;E)$ then
\begin{equation*}
\phi(\ss,p)=(\ss,p;\phi_0(\ss,p))
\end{equation*}
where $\ss\mapsto \phi_0(\ss,p)$ is a Schwartz function on $\R$ with values in $E_p$. Let $\widehat \phi$ be the section of $\p^*E$ given by
\begin{equation*}
\widehat \phi(\tau,p)=(\tau,p;\widehat \phi_0(\tau,p)),\quad \widehat \phi_0(\tau,p)=\int e^{-\im \tau\ss}\phi_0(\ss,p) d\ss.
\end{equation*}
Then $\psi\mapsto\widehat\psi$ is a continuous map $\Sch(\R\times\N;E)\to \Sch(\R\times\N;E)$ and
\begin{equation}\label{DualityRelOfFT}
\langle\widehat \phi,\psi\rangle= \langle\phi,\widehat \psi\rangle,\quad \phi\in \Sch(\R\times\N;\p^*E),\ \psi\in \Sch(\R\times\N;\p^*E^*).
\end{equation}
If $\phi\in \Sch'(\R\times\N;\p^*E)$, then $\widehat\phi$ is defined as usual by the requirement that \eqref{DualityRelOfFT} holds for all $\psi\in \Sch(\R\times\N;\p^*E^*)$.

\begin{lemma}\label{HalfSpectrumLemma}
Let $U^+$ be the interior of the set of points $p$ such that $P$ is microlocally hypoelliptic at $\nu$ for every $\nu\in \Char^+(P)\cap T^*_p\N$. Then every sequence of normalized eigenfunctions $\phi_\ell\in \mathcal E_{\tau_\ell}$ with $\tau_\ell\to\infty$ as $\ell\to\infty$ converges uniformly to zero on any compact subset of $U^+$.
\end{lemma}

\begin{proof}[Proof of Theorem~\ref{HalfSpectrum}]
Suppose that $\spec^+_0(-\im \Lie_\T)$ is an infinite set. Pick a sequence $\set{\tau_\ell}_{\ell=1}^\infty$ in $\spec^+_0(-\im \Lie_\T)$ with $\tau_\ell\to\infty$, and for each $\ell$, an element $\phi_\ell\in \mathcal E_{\tau_\ell}$ with $\|\phi_\ell\|=1$. The hypothesis in the theorem is that $U^+=\N$.  Since $\N$ is compact, the lemma gives that $\phi_\ell\to0$ uniformly on $\N$, so $\|\phi_\ell\|\to 0$, which contradicts $\|\phi_\ell\|=1$. Thus $\spec^+_0(-\im \Lie_\T)$ must be a finite set.
\end{proof}

\begin{proof}[Proof of Lemma~\ref{HalfSpectrumLemma}]
Let $K\subset U^+$ be a compact set. We argue that every subsequence of a sequence as in the lemma has a further subsequence that converges uniformly to zero on $K$. To do this, it is enough to show that if
\begin{equation}\label{SparseChoice}
\tau_{\ell+1}\geq 2\tau_\ell\quad \text{and}\quad\|\phi_\ell\|=1,
\end{equation}
then $\phi_\ell\to 0$ uniformly on $K$, since any subsequence of the original sequence has a subsequence satisfying this condition. The normalization condition in \eqref{SparseChoice} gives that the series
\begin{equation*}
\phi=\sum_{\ell=1}^\infty \phi_\ell
\end{equation*}
converges as a distribution. Indeed, from \eqref{BoundOnCoefficients} we get that $\sum_{\ell=1}^\infty \langle\phi_\ell,\psi\rangle$ converges (absolutely) for each $\psi\in C^\infty(\N;E^*\otimes |\Wedge|\N)$.

The essence of the proof is as follows. As a distribution, $\phi$ satisfies $P\phi=0$, so $\WF(\phi)\subset \Char(P)$. Since $P+(-\im\Lie_\T)^m$ is elliptic, $\Char(P)\cap \Char(-\im\Lie_\T)=\emptyset$. Therefore $\WF(\phi)$ is disjoint form the conormal bundle of any orbit $\Orb_{p_0}$ of $\T$, and consequently $\phi$ has a restriction to $\Orb_{p_0}$. Since $-\im\Lie_\T\phi_\ell=\tau_\ell\phi_\ell$,
\begin{equation*}
\a_t^*\phi(\a_t(p_0))=e^{\im\tau_\ell t}\phi_\ell(p_0).
\end{equation*}
By the continuity of the restriction map, the restriction of $\phi$ to $\Orb_{p_0}$ is the distribution
\begin{equation*}
\phi_{p_0}(t)=\sum_{\ell=1}^\infty e^{\im\tau_\ell t}\phi_\ell(p_0)
\end{equation*}
on $\R$. Since $P$ is microlocally hypoelliptic on $\Char^+(P)$, $\WF(\phi_{p_0})$ is contained in $\sym(-\im\partial_t)<0$, so if $\chi\in C^\infty_c(\R)$, then the Fourier transform of $\chi(t)\phi_{p_0}(t)$,
\begin{equation*}
f(\tau,p_0)=\sum \widehat\chi(\tau-\tau_\ell)\phi_\ell(p_0)
\end{equation*}
is rapidly decreasing in $\tau$ as $\tau\to\infty$. This can (and will) be used to prove that $\phi_\ell(p_0)\to 0$ as $\ell\to\infty$. We will show that in fact $\phi_\ell$ tends to $0$ uniformly in a neighborhood of $p_0$ in $U^+$, so by compactness of $K$ and since $p_0$ is arbitrary we will conclude that $\phi_\ell\to0$ uniformly on $K$.

Let $\tilde P$ and $\Lie_{\tilde \T}$ be the operators on sections of $\p^*E$ canonically induced by $P$ and $\Lie_\T$ via $\p$. Since $P$ commutes with $\Lie_\T$, $\tilde P \A = \A P$. This gives the first equation in
\begin{equation*}
\tilde P\A\phi=0,\quad (\Lie_{\partial_\ss}-\Lie_{\tilde\T}) \A\phi = 0,
\end{equation*}
since $P\phi=0$. The second equation is the identity \eqref{FreeIdentity}. These equations and the fact that $P$ is microlocally hypoelliptic in $\sym(-\im \T)>0$ imply
\begin{equation}\label{WFinW}
\WF(\A\phi)\subset \Char(\tilde P)\cap \Char(\Lie_{\partial_\ss}-\Lie_{\tilde \T})\cap \set{\sym(-\im \T)<0}.
\end{equation}
Let $W$ be the set on the right, a closed set. The statement \eqref{WFinW} is that $\A \phi$ belongs to the subspace $C^{-\infty}_W(\R\times\N;\p^*E)$ of elements of $C^{-\infty}(\N;\p^*E)$ whose wavefront set is contained in $W$. This subspace is a complete locally convex topological vector space, part of whose seminorms control the absence of wavefront set outside $W$ (rapid decay of the ``Fourier transform" of $\A\phi$ outside $W$), see H\"ormander \cite{Ho71a}. Let $\iota_p:\p^{-1}(p)\to \R\times \N$ be the inclusion map. The proof that the restriction map $\iota_p^*:C^{-\infty}_W(\R\times \N;\p^*E)\to C^{-\infty}_{\iota^*W}(\p^{-1}(p);\p^*E)$ is continuous involves estimating the seminorms of $C^{-\infty}_{\iota^*W}(\p^{-1}(p);\p^*E)$ expressing rapid decay outside $\iota^*W$ by the analogous seminorms for $C^{-\infty}_W(\R\times \N;\p^*E)$, see H\"ormander, op. cit. These estimates are uniform in $p$ for $p$ in small sets and give that if $\chi\in C^\infty_c(\R)$ then with $\tilde f=(\chi \A\phi)\widehat{\ }$ ($\chi$ thought of as a function of $\ss$),
\begin{equation}\label{UniformWFEstimate}
\display{300pt}{for all $p_0\in\N$ there is a neighborhood $U$ of $p_0$ such that for all $M>0$ there is $C> 0$ such that $\|(\chi \A\phi)\widehat{\ }(\tau,p)\|\leq C(1+\tau)^{-M}$ for $\tau>0$ and $p\in U$.}
\end{equation}
Lemma \eqref{FourierExpansion} implies that the series defining $\phi$ converges as a distribution, and $\A$ is continuous, so $\A\phi=\sum_\ell\A\phi_\ell$. Each $\phi_\ell$ is smooth and a solution of $-\im\Lie_\T\phi_\ell=\tau_\ell\phi_\ell$, so $\A\phi_\ell$ is smooth, equal to the section $(\ss,p)\mapsto (\ss,p;e^{\im\ss\tau_\ell}\phi_\ell(p))$. Thus $(\chi \A\phi)\widehat{\ }$ is the section $(\ss,p)\mapsto (\ss,p;f(\tau,p))$ of $\p^*E$ with
\begin{equation*}
f(\tau,p)=\sum \widehat\chi(\tau-\tau_\ell)\phi_\ell(p)
\end{equation*}
and we conclude that this function is indeed rapidly decreasing as $\tau\to\infty$ uniformly for $p$ in a neighborhood $U$ of $p_0$.

Suppose that $\int\chi(t)dt=1$, \ie, $\widehat\chi(0)=1$. Then, for each $k\in \mathbb N$,
\begin{equation}\label{Bound1}
\phi_k(p)=f(\tau_k,p)-\sum_{\ell\ne k} \widehat\chi(\tau_k-\tau_\ell)\phi_\ell(p).
\end{equation}
The fact that $f(\tau_k,0)$ tends to zero rapidly for $p\in U$ as $k\to\infty$ was established above. Using \eqref{PolynomialPointEigenBounds} and that $\widehat\chi$ is a rapidly decreasing function we bound the series as
\begin{align*}
\|\sum_{\ell\ne k}\widehat\chi(\tau_k-\tau_\ell)\phi_\ell(p)\|
&\leq
C\sum_{\ell\ne k}(1+|\tau_k-\tau_\ell|)^{-N}(1+\tau_\ell)^\mu\\
&\leq
C(1+\tau_k)^\mu \sum_{\ell\ne k}(1+|\tau_k-\tau_\ell|)^{-N+\mu}
\end{align*}
with arbitrary $N$. Using the integral test and the condition on the $\tau_\ell$ in \eqref{SparseChoice} one obtains the bound
\begin{equation*}
\sum_{\ell\ne k}(1+|\tau_k-\tau_\ell|)^{-N+\mu}\leq C(\tau_{k-1}^{-N+\mu+1}+\tau_k^{-N+\mu+1}),
\end{equation*}
so the norm (as an element of $E_p$) of the series in \eqref{Bound1} is rapidly decreasing, uniformly for $p\in U$. Fix $N>\mu+1$. We conclude that if \eqref{UniformWFEstimate} holds in $U$ with with $M=1$, then
\begin{equation*}
\|\phi_k(p)\|\leq C(1+\tau_k)^{-1} \text{ for all }p\in U.
\end{equation*}
The compactness of $K$ gives that the same conclusion is valid for all $p\in \N$ (with some other constant). This implies that the pointwise norm of the $\phi_k$ tends to $0$ uniformly on $K$ as $k\to\infty$.
\end{proof}

\section{CR manifolds with $\R$-action}\label{sCRmanifolds}

Let $\N$ be a CR manifold, write $\Kbar$ for its CR structure (as complex tangent vectors of type $(0,1)$), let $\Hor\subset T\N$ be the subbundle whose complexification is $\K\oplus \Kbar$ and write $J:\Hor\to\Hor$ for the almost complex structure of $\Hor$.

Assume $\T$ (as usual, a nowhere vanishing real vector field) is such that $d\a_t$ maps $\Kbar$ to itself, that is, $\a_t$ acts by CR diffeomorphisms. Equivalently, $[\X,\T]$ is a smooth vector field in $\Kbar$ whenever $\X$ is. It follows that
\begin{equation}\label{Veebar}
\Veebar=\Kbar+\Span_\C\T
\end{equation}
is an involutive subbundle of $\C T\N$. Since $\Vee+\Veebar=\C T\N$, a theorem of Nirenberg in \cite{Ni57} extending the Newlander-Nirenberg Theorem, see \cite{NeNi57}, implies that $\N$ is locally integrable, that is to say, locally embeddable or realizable. (More general results of this nature were obtained by  Baouendi-Rothschild \cite{BR87}, Baouendi-Rothschild-Treves \cite{BRT85}, and Jacobowitz \cite{Jacob87}.)

If $\theta$ is the $1$-form that vanishes on $\K\oplus\Kbar$ and satisfies $\langle\theta,\T\rangle=1$, then $\T$ is a Reeb vector field with respect to $\theta$ and the latter is, if $\Kbar$ is non-degenerate, a pseudohermitian structure on $\N$, see Webster \cite{Webster}.

We assume in addition that there is a Riemannian metric $g$ on $\N$ which is $\T$-invariant: $\Lie_\T g=0$. Then $g(\T,\T)$ is constant on integral curves of $\T$, so we may normalize $g$ so as to also have that $g(\T,\T)=1$. The restriction of $g$ to $\Hor$ is also $d\a_t$-invariant. Redefine $g$ so that $\Hor$ is orthogonal to $\T$, finally, replace $g$ on $\Hor$ by the metric
\begin{equation*}
(u,v)\mapsto \frac{1}{2}\big(g(u,v)+g(Ju,Jv)\big),\quad u,v\in T_p\N,\ p\in \N
\end{equation*}
Since $d\a_t J=Jd\a_t$, the new metric is again $\T$-invariant in addition to Hermitian. So it gives a $\T$-invariant Hermitian metric on $\C T\M$ making $\K$, $\Kbar$ and $\Span_\C\T$ orthogonal to each other. Conversely, a $\T$-invariant Hermitian metric on $\Kbar$ can be used to construct a $\T$-invariant Riemannian metric with respect to which these subbundles are orthogonal to each other. 

\medskip
We regard $\Veebar$, defined in \eqref{Veebar} as the primary object together with  a class $\pmb \beta$ of sections of $\Veebar$ to be defined momentarily. Because $\Veebar$ is involutive, there is a complex
\begin{equation}\label{bdyComplex}
\cdots
\to C^\infty(\N;\Wedge^q \smash[t]{\Veebar}^*)\xrightarrow{\Deebar}
C^\infty(\N;\Wedge^{q+1}\smash[t]{\Veebar}^*)
\to\cdots,
\end{equation}
where $\Deebar$ is defined using Cartan's formula for the standard differential, see Helgason \cite{He}. Namely, if $\eta\in C^\infty(\N;\Wedge^q \smash[t]{\Veebar}^*)$ and $V_0,\dotsc,V_q$ are smooth sections of $\Veebar$, then
\begin{multline*}
(q+1)\Deebar \eta(V_0,\dotsc,V_q) = \sum_j (-1)^jV_j \eta(V_0,\dotsc,\hat V_j,\dotsc,V_q)\\+
\sum_{j<k}(-1)^{j+k}\eta([V_j,V_k],V_1,\dotsc,\hat V_j,\dotsc,\hat V_k,\dotsc,V_q).
\end{multline*}
The complex \eqref{bdyComplex} is elliptic because $\Vee+\Veebar=\C T\N$ see Treves \cite{Tr92}. For a function $f$ we have $\Deebar f=\iota^*df$, where $\iota^*:\C T^*\N\to \smash[t]{\Veebar}^*$ is the dual of the inclusion homomorphism $\iota:\Veebar\to\C T\N$. 

The form $\beta=-\im \theta\big|_{\Veebar}$ is an element of $C^\infty(\N;{\Veebar}^*)$ and 
\begin{gather*}
2\Deebar \beta(\X,\Y)=-\im\X\theta(\Y)+\im\Y\theta(X)+\im \theta([\X,\Y])=0,\\ 2\Deebar \beta(\X,\T)=-\im \X\theta(\T)+\T\theta(X)+\im\theta([X,\T])=0
\end{gather*}
if $\X$, $\Y$ are sections of $\Veebar$, so $\Deebar\beta=0$. We let
\begin{equation*}
\pmb \beta=\set{\beta+\Dee u:u\in C^\infty(\N,\R),\ \T u=0}
\end{equation*}
Each element $\beta'\in \pmb\beta$ is $\Deebar$-closed and $\langle \beta,\T\rangle =-\im$, therefore 
\begin{equation*}
\Kbar_{\beta'}=\ker\beta'
\end{equation*}
is again CR structure. Since $\Lie_\T\beta'=0$, these CR structures are all $\T$-invariant.

The meaning $\Veebar$ together with the class $\pmb \beta$ is illustrated in Section~\ref{sLineBundles}, see \eqref{VeebarInCircleBundle} and the end of that section. For an interpretation of the condition $\Deebar\beta=0$ in a familiar situation see Lemma~\ref{HolomorphicConnection}.

\medskip
Fix an element in $\pmb\beta$. In terms of basic properties there is no distinction between any of the elements  of $\pmb\beta$, so we continue to denote our choice by $\beta$, and by $\Kbar$ the CR structure it defines. The operators of the CR complex
\begin{equation}\label{dbarbComplex}
\cdots \to C^\infty(\N;\Wedge^q\Kbar^*)\xrightarrow{\deebarb} C^\infty(\N;\Wedge^{q+1}\Kbar^*)\to\cdots
\end{equation}
can be written in terms of those of the complex \eqref{bdyComplex}:
\begin{equation}\label{FormulaForDeeOnK*}
\deebarb\phi=\Deebar\phi-\im \beta\wedge \Lie_\T\phi,\quad\phi\in C^\infty(\N;\Wedge^q\Kbar^*).
\end{equation}
Here $\Lie_\T$ means regular Lie derivative. The $\T$ invariance of $\beta$ (hence of $\Kbar$) also gives that $\deebarb$ commutes with $\Lie_\T$. If $h$ is a Hermitian metric on $\Kbar$ which is $\T$-invariant and $\m$ a positive $\T$-invariant density (for instance the Riemannian density), then also the formal adjoint, $\deebarb^\star$, of $\deebarb$, is $\T$ invariant (since $-\im\Lie_\T$ is symmetric), hence the Hodge Laplacians of the $\deebarb$-complex,
\begin{equation*}
\Laplacian_{b,q} = \deebarb\deebarb^\star + \deebarb^\star\deebarb
\end{equation*}
are also $\T$-invariant. 

We are now ready to apply the results of Section~\ref{sInvariantOperators}. Let
\begin{equation*}
\Ha^q_{\deebarb}(\N)=\ker\Laplacian_{b,q}=\set{\phi\in L^2(\N;\Wedge^q\Kbar^*):\Laplacian_{b,q}\phi=0}
\end{equation*}
and let
\begin{equation*}
\Dom_q=\set{\phi\in \Ha^q_{\deebarb}(\N)\text{ and }\Lie_\T\phi\in \Ha^q_{\deebarb}(\N)}.
\end{equation*}
The spaces $\Ha^q_{\deebarb}(\N)$ may be infinite-dimensional. If $\phi\in \Ha^q_{\deebarb}(\N)$, the condition $\Lie_\T\phi\in \Ha^q_{\deebarb}(\N)$ is equivalent to the condition
\begin{equation*}
\Lie_\T\phi\in L^2(\N;\Wedge^q\Kbar^*).
\end{equation*}
Since $\Laplacian_{b,q}-\Lie_\T^2$ is elliptic and symmetric, \eqref{DoubleRayCondition} is satisfied for any real line $\Lambda\subset \C$ whose only real point is $0$. Theorem~\ref{TIsFredholm} gives:

\begin{theorem}\label{CRDecomposition}
Suppose that there a $\T$-invariant Hermitian metric $h$ on $\Kbar$, let $\Laplacian_{b,q}$ be the Laplacian of the complex \eqref{dbarbComplex} computed using the metric $h$ and a $\T$-invariant density on $\N$. Then
\begin{equation}\label{LieOnB-Harmonic}
-\im \Lie_\T\big|_{\Dom_q}:\Dom_q\subset \Ha^q_{\deebarb}(\N)\to \Ha^q_{\deebarb}(\N)
\end{equation}
is a selfadjoint Fredholm operator with compact resolvent.
\end{theorem}

\begin{definition}
Let $\spec^q_0(-\im \Lie_\T)$ be the spectrum of the operator \eqref{LieOnB-Harmonic}, and let $\Ha^q_{\deebarb,\tau}(\N)$ be the eigenspace of $-\im\Lie_\T$ in $\Ha^q_{\deebarb}(\N)$ corresponding to the eigenvalue $\tau$.
\end{definition}


\section{Vanishing theorems}\label{sVanishing}

We continue in this section with the notation and assumptions of the previous section and give an application of Theorem~\ref{HalfSpectrum} when the CR structure $\Kbar$ is non-degenerate. Let $\Char \Kbar$ be the characteristic set of $\Kbar$ and let
\begin{equation*}
\Char^\pm\Kbar=\set{\nu\in\Char \Kbar:\pmb \tau(\nu)\gtrless 0}
\end{equation*}
where $\pmb \tau=\sym(-\im\T)$. Let $\theta$ be the real $1$-form on $\N$ which vanishes on $\Kbar$ and satisfies $\langle\theta,\T\rangle =1$; thus $\theta$ is smooth, spans $\Char \Kbar$, and has values in $\Char^+(\Kbar)$. Recall that
\begin{equation*}
\Levi_{\theta}(v,w)=-\im d\theta(v,\overline w), \quad v,\ w\in \K_p,\ p\in \N.
\end{equation*}
Suppose that $\Levi_{\theta}$ is non-degenerate, with $k$ positive and $n-k$ negative eigenvalues. It is well known that then $\Laplacian_{b,q}$ is microlocally hypoelliptic at $\nu\in\Char\K$ for all $q$ except if $q=k$ and $\pmb \tau(\nu)<0$ or if $q=n-k$ and $\pmb \tau(\nu)>0$, see \cite{BdM,Sj}, also the appendix of \cite{Me3}. In the definition of the Levi form above we switched to from $\Kbar$ to $\K$ to adapt to the usual conventions.

Applying Theorem~\ref{HalfSpectrum} we get:

\begin{theorem}\label{WeakVanishing}
With the hypotheses of Theorem~\ref{CRDecomposition}, suppose that $\Levi_{\theta}$ is non-degenerate with $k$ positive and $n-k$ negative eigenvalues. Then
\begin{enumerate}
\item $\spec_0^q(-\im \Lie_\T)$ is finite if $q\ne k,\ n-k$;
\item $\spec_0^k(-\im\Lie_\T)$ contains only finitely many positive elements, and \item $\spec_0^{n-k}(-\im\Lie_\T)$ contains only finitely many negative elements.
\end{enumerate}
\end{theorem}

For the interpretation of this result in the light of Kodaira's vanishing theorem, see Section~\ref{sLineBundles}, in particular \eqref{SpectrumAndKodairaVanishing}. Theorem~\ref{WeakVanishing} applied to the case where $\N$ is the circle bundle of a Hermitian holomorphic line bundle $E\to\B$ over a compact manifold is a partial version of various theorems on vanishing of the $\deebar$-cohomology with coefficients in $E$, see for instance Kobayashi \cite[Chapter III, \S 3]{Ko87} for a listing of such theorems.

One can make a stronger statement when $q=0$. The condition $\zeta\in \Ha^0_{\deebarb}(\N)$ just means that $\deebarb\zeta=0$. For such $\zeta$, if $-\im \T\zeta=\tau\zeta$, then $\zeta$ is smooth and for each $\ell\in \mathbb N$, $\zeta^\ell \in \Ha^0_{\deebarb}(\N)$ satisfies $-\im\T\zeta^\ell=\ell\tau\zeta^\ell$. So if for instance $\spec_0^0(-\im \Lie_\T)\cap \R_+$ is a finite set, then in fact $\spec_0^0(-\im \Lie_\T)$ contains no positive elements. In particular, with the hypothesis of the theorem, if $k$ and $n-k$ are different from $0$, then $\spec_0^0(-\im \Lie_\T)=\set{0}$ rather than just finite.

\section{Circle bundles of holomorphic line bundles}\label{sLineBundles}

We will now discuss circle bundles of holomorphic line bundles in the context of the preceding sections.

\medskip
Let $\B$ be a manifold, let $E\to\B$ be a complex line bundle and fix a Hermitian metric on $E$. Let $\rho:SE\to\B$ be the circle bundle. For $m\in \Z$ define the tensor product bundles $E^m\to\B$ in the usual way, give each of these line bundles the Hermitian metric induced by that of $E$ and let $SE^m\to\B$ be the circle bundle.

Define $\wp_m:SE\to SE^m$ for $m \ne 0$ as follows. Let $p\in SE$. If $m>0$, let $\wp_m(p)=p\otimes\dotsm\otimes p$ ($m$ times). If $m<0$, let $p^*\in SE^*$ be the element dual to $p$, and let $\wp_m(p)=p^*\otimes \dotsm \otimes p^*$ ($|m|$ times). The map $\wp_m:SE\to SE^m$ is an $|m|$-sheeted covering map with the property that $\wp_m(e^{\im t}p)=e^{\im m t}\wp_m(p)$.

Let $x\in \B$. A point $\eta\in E^m_x$ is a linear function $\eta:E^{-m}_x\to\C$ which as such gives the function $f_\eta=\eta\circ \wp_{-m}:SE_x\to\C$. The latter function has the property that if $p\in SE_x$, then $f_\eta(e^{\im t}p)=e^{-\im m t}f_\eta(p)$. Thus if $\T$ is the infinitesimal generator of the action of $S^1$ on $SE$, the function $f_\eta$ on $SE_x$ satisfies the equation
\begin{equation}\label{FiberOfEm}
\T f+\im m f=0.
\end{equation}
Conversely, if $f:E_x\to\C$ solves this equation, then $f$ is the pullback to $SE_x$ by $\wp_{-m}$ of a unique function $\eta_f:SE^{-m}_x\to\C$ that satisfies $\eta_f(e^{\im t}p')=e^{\im t}\eta_f(p')$, $p'\in SE^{-m}_x$, and that therefore extends as a linear map $\eta_f:E^{-m}_x\to\C$ thus giving an element of $E_x^m$. The correspondence $f\mapsto \eta_f$ is the inverse of $\eta\mapsto f_\eta\circ\wp_{-m}$: the fiber $E^m_x$ is isomorphic, as a vector space, to the space of solutions of \eqref{FiberOfEm} on $SE_x$.

More generally, if $x\in \B$ and $\eta\in \Wedge^q_x\B \otimes E^m_x$, then $\langle\wp_{-m}(p),\eta\rangle\in T_x\B$ for each $p\in SE_x$ and $\rho^*_p\langle\wp_{-m}(p),\eta\rangle$ is an element of $\Wedge^q_p SE$. There is a canonical identification of $\rho^*T^*\B$ and the kernel, $\Hor^*$, of $\inner_\T:\Wedge^q SE \to \Wedge^{q-1}SE$ (interior multiplication by $\T$), and the map
\begin{equation*}
SE_x \ni p\mapsto
F_m(\eta)(p) = \rho^*_p\langle\wp_{-m}(p),\eta\rangle \in \Wedge^q_p\Hor^*
\end{equation*}
is a section $\phi_\eta$ of $\Wedge^q\Hor^*$ along $SE_x$.
Since
\begin{equation}\label{BundlePeriodicity}
\begin{aligned}
F_m(\eta)(\a_t p)
&= e^{-\im m t}\rho_{\a_t p}^*\langle\wp_{-m}(p),\eta\rangle\\
&= e^{-\im m t} \a_{-t}^*\rho^*_p\langle\wp_{-m}(p),\eta\rangle\\
&= e^{-\im m t} \a_{-t}^*(F_m(\eta)(p)),
\end{aligned}
\end{equation}
$\a_t^* (F_m(\eta)(\a_tp))=e^{-\im m t}F_m(p)$, so $\phi=F_m(\eta)$ satisfies
\begin{equation}\label{FiberOfEmOtimesWedge}
\Lie_\T \phi +\im m \phi=0.
\end{equation}
Conversely, for any section $\phi$ of $\Wedge^q\Hor^*$ along $SE_x$ that satisfies \eqref{FiberOfEmOtimesWedge} there is $\eta\in \Wedge^q_x\B\otimes E^m_x$ such that $\phi=F_m(\eta)$. Applying this to sections of $\Wedge^q\B\otimes E^m$ we get an injective map
\begin{equation}\label{wpStar}
F_m:C^\infty(\B;\Wedge^q\B\otimes E^m)\to C^\infty(SE;\Wedge^q\Hor^*)
\end{equation}
whose range is the subspace of $C^\infty(SE;\Wedge^q\Hor^*)$ whose elements satisfy \eqref{FiberOfEmOtimesWedge} globally. The case $m=0$ is included in the above scheme by defining $F_0=\rho^*$. We give $E^0=\B\times \C$ the canonical Hermitian structure.

Suppose that $\nabla$ is a Hermitian connection on $E$. Thinking of $SE$ as the bundle of unit bases of $E$, the connection gives a horizontal bundle $\Hor_\theta\subset T SE$ and a connection form $\theta$; $\theta$ vanishes on $\Hor_\theta$, $\langle\theta,\T\rangle=1$, and $\Lie_\T\theta=0$. Via the splitting $T SE=\Hor_\theta \oplus \Span \T$, the dual of $\Hor_\theta$ is identified with $\Hor^*$, and $\Wedge^q\Hor_\theta^*$ is identified with the kernel $\Wedge^q\Hor^*$ of $\inner_\T:\Wedge ^qSE\to\Wedge^{q-1}SE$.

\medskip
Suppose now that $\B$ is a complex manifold and that $\pi:E\to \B$ is a Hermitian holomorphic line bundle. Let $\nabla$ be the Hermitian holomorphic connection, view $\N=SE$, the circle bundle of $E$ with respect to the metric as the unit frame bundle,  let $\theta$ be the connection form of $\nabla$. Let $\iota:\Veebar\embed \C T\N$ be the subbundle of $\C T\N$ given by
\begin{equation}\label{VeebarInCircleBundle}
\Veebar=\set{v\in \C T\N:\pi_*v\in T^{0,1}\B}
\end{equation}
and let $\beta=-\im\iota^*\theta$. Denote the operators of the associated differential complex by $\Deebar$, as usual. The kernel $\Kbar_\beta\subset \Veebar$ of $\beta$ is $\T$-invariant, equal to $\Veebar\cap\ker \theta\subset \C\Hor_\theta$; its fibers project isomorphically onto the fibers of $T^{0,1}\B$. The kernel, $\Kbar^*$, of $\inner_\T$ in $\smash[t]{\Veebar}^*$ is canonically isomorphic to $\rho^*\Wedge^{0,1}\B$ and to the dual of $\Kbar_\beta$.

\begin{lemma}\label{HolomorphicConnection}
The section $\beta$ is $\Deebar$-closed.
\end{lemma}
\begin{proof}
Suppose $V$ and $W$ are smooth $\T$-invariant vector fields in $\Kbar_\beta$. Then
\begin{multline*}
2\Deebar\beta(V,W)=V\langle\beta,W\rangle-W\langle\beta,V\rangle-\langle\beta,[V,W]\rangle
\\=-\im V\langle\theta,W\rangle + \im W\langle\theta,V\rangle + \im\langle\theta,[V,W]\rangle
= -2\im d\theta(V,W) = -2\rho^*\Omega(V,W)
\end{multline*}
where $\Omega$ is the curvature form of the connection. Since the latter is a holomorphic connection, its $(0,2)$ component vanishes. Thus, since at each point $V$ and $W$ are liftings of elements of $T^{0,1}\B$, $\rho^*\Omega(V,W)=0$. Also
\begin{equation*}
2\Deebar\beta(V,\T)=V\langle\beta,\T\rangle-\T\langle\beta,V\rangle-\langle\beta,[V,\T]\rangle=0.
\end{equation*}
Thus $\Deebar\beta=0$.
\end{proof}

Evidently, the map \eqref{wpStar} restricts to an isomorphism from $C^\infty(\B;\Wedge^{0,q}\B\otimes E^m)$ onto
\begin{equation*}
C_m^\infty(SE;\Wedge^q\Kbar^*)=\set{\phi\in C^\infty(SE;\Wedge^q\Kbar^*):\Lie_\T\phi+\im m\phi=0}.
\end{equation*}

\begin{lemma}
The map
\begin{equation*}
F_m:C^\infty(\B;\Wedge^{0,q}\B\otimes E^m)\to C_m^\infty(SE;\Wedge^q\Kbar^*)
\end{equation*}
is an isomorphism, and
\begin{equation}\label{DeebarDbarHomotopy}
\Dbar(\im m)F_m=F_m\deebar
\end{equation}
where $\Dbar(\sigma)\phi= \Deebar\phi +\im \sigma \beta\wedge \phi$.
\end{lemma}

\begin{proof}
We prove \eqref{DeebarDbarHomotopy}. Let $\gamma$ be a smooth local section of $SE$ defined near a point $x_0\in \B$ and let $\omega$ be the connection form  with respect to $\gamma$ of the Hermitian holomorphic connection of $E$. Then $\gamma_m=\wp_m\circ \gamma$ is a section of $SE^m$, and $m\omega$ is the connection form with respect to $\gamma_m$ of the Hermitian holomorphic connection of $E^m$. If $\eta=\phi\otimes \gamma_m$ is a smooth section of $\Wedge^{0,q}\B \otimes E^m$ near $x_0$, then
\begin{equation*}
\deebar\eta= ( m\omega^{0,1}\wedge\phi+\deebar\phi) \otimes \gamma_m
\end{equation*}
where $\omega^{0,1}$ is the $(0,1)$ component of $\omega$, and $F_m(\eta)(\gamma(z))=d\rho^*_{\gamma(z)}\phi$.

Let $t$ be defined in an neighborhood of $p_0=\gamma(x_0)$ in $SE$ so that $t$ vanishes on the image of $\gamma$ and $\T t=1$. Then $\theta=dt-\im \rho^*\omega$ and
\begin{equation*}
\beta=-\im \Deebar t-\rho^*\omega^{0,1}.
\end{equation*}
Suppose that $z^1,\dotsc,z^n$ are holomorphic coordinates for $\B$ on $U$. Their pullback to $\rho^{-1}(U)$ will also be denoted $z^1,\dotsc,z^n$. Then $(z^1,\dotsc,z^n,t)$ is a hypoanalytic chart of $\Veebar$ near $p_0$ (see Treves \cite{Tr92}), and $\a_t \gamma(z) = e^{\im t}\gamma(z)$ is the point with coordinates $(z,t)$. In these coordinates, if $\phi=\sum_{I}\phi_I d\overline z^I$, then $\rho^*\phi=\sum_{I}\phi_I\Deebar \overline z^I$. Using \eqref{BundlePeriodicity} we have
\begin{equation*}
F_m(\eta)(\a_t\gamma(z)) = e^{-\im m t}\sum_{I}\phi_I \Deebar\overline z^I
\end{equation*}
so
\begin{align*}
\Deebar F_m(\eta)(\a_t \gamma(z))
&=e^{-\im m t} \big(\sum_{I}\Deebar \phi_I\wedge\Deebar\overline z^I - \im m \Deebar t\wedge \sum_{I}\phi_I \Deebar\overline z^I\big)\\
&=e^{-\im m t}\big(\sum_{I}\Deebar \phi_I\wedge\Deebar\overline z^I+ m\rho^*\omega^{0,1}\wedge \sum_{I}\phi_I \Deebar\overline z^I\big) \\
&\quad + e^{-\im m t}\big(- m\rho^*\omega^{0,1}\wedge \sum_{I}\phi_I \Deebar\overline z^I-\im m\Deebar t\wedge \sum_{I}\phi_I \Deebar\overline z^I \big)\\
&=e^{-\im m t}\big(\rho^*(\deebar\phi_0+m\omega^{0,1}\wedge\eta)+m(-\im \Deebar t-\rho^*\omega^{0,1})\wedge \rho^*\eta\big)\\
&=F_m(\deebar\eta)(\a_t \gamma(z))+m\beta\wedge F_m(\eta)(\a_t \gamma(z)).
\end{align*}
Thus $\Deebar F_m(\eta) - m\beta\wedge F_m(\eta) = F_m(\deebar \eta)$.
\end{proof}

The vector bundle $\Kbar_\beta$ is a CR structure on $SE$. Using the identification of $\Kbar_\beta^*$ and $\Kbar^*$ indicated above, the $\deebarb$-operators of this CR structure are given by \eqref{FormulaForDeeOnK*}. Since $\Lie_\T\deebarb=\deebarb\Lie_\T$, there is a complex
\begin{equation}\label{PeriodicDeebarb}
\cdots\to C_m^\infty(SE;\Wedge^q\Kbar^*) \xrightarrow{\deebarb}C_m^\infty(SE;\Wedge^{q+1}\Kbar^*)\to\cdots
\end{equation}
for each $m\in \Z$.

\begin{lemma}
The maps $F_m$ satisfy
\begin{equation*}
\deebarb F_m=F_m\deebar.
\end{equation*}
Hence, the $\deebar$ cohomology groups of $E^m$ are isomorphic to the cohomology groups of the complex \eqref{PeriodicDeebarb}.
\end{lemma}

Indeed, suppose $\eta\in C^\infty(\B;\Wedge^{0,q}\B\otimes E^m)$. Then \eqref{FiberOfEmOtimesWedge} holds for $\phi=F_m(\eta)$. With this we get
\begin{equation*}
\deebarb F_m(\eta) = \Deebar F_m(\eta)-\im\beta\wedge \Lie_\T F_m(\eta) = \Deebar F_m(\eta) - m\beta\wedge F_m(\eta)=F_m(\deebar\eta).
\end{equation*}
where the last equality is \eqref{DeebarDbarHomotopy}.

\medskip
Fix a Hermitian metric $g$ on $\Wedge^{0,1}\B$ and denote also by $g$ the Riemannian metric it induces on $\B$ as well as those induced on each of the exterior powers $\Wedge^{0,q}\B$. Let $\m$ be the Riemannian density determined by $g$. Let $h$ be the Hermitian metric of $E$ and $h_m$ the one induced on $E^m$. So $h_m(\wp_m(p),\wp_m(p))=h(p,p)=1$ if $p\in SE$. These metrics give Hermitian metrics on the vector bundles $\Wedge^{0,q}\B\otimes E^m$, which we again denote by $h$:
\begin{equation*}
h(\phi\otimes \eta,\psi\otimes \eta)=g(\phi,\psi)\quad \text{if } h(\eta,\eta)=1.
\end{equation*}
Let $\Laplacian^{(m)}$ denote the Hodge Laplacians for the Dolbeault complex with coefficients in $E^m$.

Using $g$ and the pointwise isomorphisms $\rho_p^*:\Wedge^{0,1}_{\rho(p)}\B\to \Wedge^q_p\Kbar^*$ we get Hermitian metrics on the vector bundles $\Wedge^q\Kbar^*$, to be denoted again by $h$. Using this Hermitian metric and the density $\m_0 =\rho^*\m\otimes \theta$ on $SE$ we then get Kohn Laplacians $\Laplacian_b$ for the $\deebarb$ complex on $SE$. It follows from the definitions that
\begin{equation*}
\T h(\phi,\psi)=h(\Lie_\T\phi,\psi)+h(\phi,\Lie_\T\psi),
\end{equation*}
if $\phi$, $\psi\in C^\infty(SE;\Wedge^q\Kbar^*)$. This gives that $-\im \Lie_\T$ is formally selfadjoint,
\begin{equation*}
\int_{SE}h(-\im\Lie_\T\phi,\psi)\m_0 = \int_{SE}h(\phi,-\im\Lie_\T\psi)\m_0
\end{equation*}
and that the spaces $C_m^\infty(SE;\Wedge^q\Kbar^*)$, $m\in\Z$, are pairwise orthogonal. Also, since $\Lie_\T$ commutes with $\deebarb$, $\Lie_\T$ commutes with $\Laplacian_b$.

The formula
\begin{equation*}
h((F_m\phi)(p),(F_m\psi)(p))=h(\phi(\rho(p)),\psi(\rho(p)))\quad \text{if }\phi,\ \psi\in C^\infty(\B;\Wedge^{0,q}\B\otimes E^m)
\end{equation*}
for any $p\in SE$ gives
\begin{equation*}
\int_{SE}h(F_m\phi,F_m\psi)\m_0=2\pi \int_\B h(\phi,\psi)\m \quad \text{if }\phi,\ \psi\in C^\infty(\B;\Wedge^{0,q}\B\otimes E^m).
\end{equation*}
Now, if $\phi\in C^\infty(\B;\Wedge^{0,q}\B\otimes E^m)$ and $\psi\in C^\infty(\B;\Wedge^{0,q+1}\B\otimes E^m)$, then
\begin{equation*}
(\deebarb F_m \phi,F_m\psi)=(F_m(\deebar \phi),F_m\psi) = 2\pi(\deebar\phi,\psi)=(F_m\phi,F_m(\deebar^\star\psi))
\end{equation*}
and also $(\deebarb F_m \phi,F_m\psi) = ( F_m \phi,\deebarb^\star F_m\psi)$, so
\begin{equation*}
\deebarb^\star F_m = F_m\deebar^\star
\end{equation*}
Consequently $\Laplacian_bF_m=F_m\Laplacian_m$, and thus the kernel of $\Laplacian^{(m)}$ in degree $q$ is mapped by $F_m$ into the subspace $\mathcal E^q_m$ of the kernel of $\Laplacian_b$ in degree $q$ whose elements satisfy \eqref{FiberOfEmOtimesWedge}. Let $\ker \Laplacian_b$ be the kernel of $\Laplacian_b$ in $L^2$. Let
\begin{equation*}
\Dom=\set{\phi\in \ker\Laplacian_b:\Lie_\T\phi\in L^2}.
\end{equation*}
Theorem~\ref{TIsFredholm} gives here that
\begin{equation}\label{MiniSAT}
-\im \Lie_\T\big|_\Dom:\Dom\subset \ker\Laplacian_b\to \ker\Laplacian_b
\end{equation}
is a selfadjoint Fredholm operator with eigenspaces consisting of smooth sections. In the situation at hand, the numbers $\tau$ for which
\begin{equation*}
\Laplacian_b \phi=0, \quad-\im \Lie_\T\phi=\tau\phi
\end{equation*}
has a nontrivial solution must be integers. This gives, for each $q$, an isomorphism between the eigenspace of \eqref{MiniSAT} corresponding to the eigenvalue $-m$ and the kernel of $\Laplacian^{(m)}$ on $(0,q)$ forms.
\begin{equation}\label{SpectrumAndKodairaVanishing}
\display{300pt}{Theorems on vanishing of cohomology in degree $(0,q)$ for the Dolbeault complex associated with $E^m$ are thus theorems on absence of the point $-m$ from the spectrum of the operator \eqref{MiniSAT} in degree $q$.}
\end{equation}
For example, if $E$ is a positive line bundle, then by Kodaira's Vanishing Theorem, $-m\notin\spec(-\im\Lie_\T)$ for every $q<n$ and $m\geq 1$. Positivity (or ampleness) of $E$ means that for \emph{some} Hermitian metric, $SE$ is strictly pseudo{\em concave} (Grauert), which in turn implies microlocal hypoellipticity of $\Laplacian_b$ on one component or the other of its characteristic set, depending on the degree. Theorem~\ref{HalfSpectrum} gives a sufficient condition in terms of a hypoellipticity condition of $\Laplacian_b$ in order for the spectrum of \eqref{MiniSAT} to contain only finitely many points in one of the components of $\R\minus 0$. The simplest sufficient condition for hypoellipticity of $\Laplacian_b$ is the nondegeneracy of the Levi form of the CR structure $\K$. In that case, Theorem~\ref{HalfSpectrum} gives our Theorem~\ref{WeakVanishing}.

\medskip
We now discuss the class $\pmb \beta$. The forms $\beta$, $\beta'$ on the circle bundles $\rho:SE\to\B$ and $\rho':S'E\to\B$ of $E$ with respect to two Hermitian metrics $h$ and $h'$ are related as follows. Let $u$ be the function $\B\to\R$ such that $h'=e^{2u} h$. Let $\theta$, $\theta'$ be the connection forms of the respective Hermitian holomorphic connections as forms on the respective circle bundles. Finally, let $F:SE\to S'E$ be the map
\begin{equation*}
F(\sigma)=e^{-u}\sigma.
\end{equation*}
Then \begin{equation}\label{RelationBetweenConnections}
F^*\theta'=\theta-\im(\dee u -\deebar u).
\end{equation}
The map $F$ is not a CR map, but since $\rho'\circ F=\rho$, its differential maps the structure bundle $\Veebar$ of $SE$ to the structure bundle $\smash[t]{\Veebar}'$ of $S'E$, and \eqref{RelationBetweenConnections} gives
\begin{equation}\label{RelationBetweenBetas}
F^*\beta'=\beta + \Deebar u.
\end{equation}
More explicitly, let $\sigma$ be a holomorphic frame of $E$ over some open set $U\subset \B$. Let $\sigma_0=\sigma/|\sigma|$. The pull-back of the connection form $\theta$ to $U\times S^1$ by the diffeomorphism $\Phi=U\times S^1 \to \rho^{-1}(U)$ given by $\Phi(x,e^{\im t})= e^{\im t}\sigma_0(x)$ is
\begin{equation*}
\Phi^*\theta = dt -\im \frac{\dee |\sigma|^2 - \deebar|\sigma|^2}{2|\sigma|^2}
\end{equation*}
where $|\sigma|=\sqrt{h(\sigma,\sigma)}$. Similarly, with the Hermitian holomorphic connection determined by $h'$ and the analogously defined map $\Phi':U\times S^1\to {\rho'}^{-1}(U)$ we get
\begin{equation*}
{\Phi'}^*\theta' = dt -\im \frac{\dee (e^{2u}|\sigma|^2) - \deebar (e^{2u}|\sigma|^2)}{2e^{2u}|\sigma|^2}
=dt -\im \bigg(\frac{\dee |\sigma|^2 - \deebar|\sigma|^2}{2|\sigma|^2} + \dee u- \deebar u\bigg).
\end{equation*}
This proves \eqref{RelationBetweenConnections} since $\Phi'=F\circ \Phi$, and restricting both sides of \eqref{RelationBetweenConnections} to $\Vee$ and multiplying by $-\im$ gives \eqref{RelationBetweenBetas}. Thus the class $\pmb \beta$ includes the forms $\beta$ defined by any Hermitian metric on $E$.

\end{document}